\documentclass[12pt]{article}
\usepackage{amsmath,amsthm,amssymb,amsfonts,authblk,float,url}
\title{On the density of nice Friedmans}
\author{Michael Brand\\
\texttt{michael.brand@alumni.weizmann.ac.il}}
\affil{Monash University School of IT\\
Clayton, VIC 3800\\
Australia}
\date{\today}

\renewcommand{\span}[1]{\mathop{\mathrm{span}}\!\left(#1\right)}

\newcommand{\spant}[1]{\mathop{\mathrm{span'}}\!\left(#1\right)}
\newcommand{\tuplespant}[1]{\mathop{\mathrm{tuplespan'}}\!\left(#1\right)}
\newtheorem{thm}{Theorem}
\newtheorem{lemma}{Lemma}

\begin{document}
\maketitle
\begin{abstract}
A Friedman number is a positive integer which is the result of an
expression combining all of its own digits by use of the four basic operations,
exponentiation and digit concatenation. A ``nice'' Friedman number is a
Friedman number for which the expression constructing the number from its
own digits can be represented with the original order of the digits
unchanged. One of the fundamental questions regarding Friedman numbers, and
particularly regarding nice Friedman numbers,
is how common they are among the integers. In this
paper, we prove that nice Friedman numbers have density 1,
when considered in binary, ternary or base four.
\end{abstract}

\section{Introduction}
Friedman numbers \cite{A036057:FriedmanNumbers, Friedman:PuzzlePage}
are numbers that can be computed from their own digits, each
digit used exactly once, by use
of the four basic arithmetic operations, exponentiation and digit concatenation
(as long as digit concatenation is not the only operation used).
Parentheses can be used at will.
An example of a Friedman number is $25$, which can be represented as $5^2$.
An example of a non-Friedman number is any power of $10$, because no power
of $10$ can be expressed as the result of a computation using only arithmetic
operations and exponentiation if the initial arguments in the computation are
a smaller power of $10$ and several zeros.

Several interesting subsets of Friedman numbers have been defined since
the introduction of Friedman numbers. For example,
there are several conflicting definitions in the literature
(see, e.g., \cite{OEIS:vampire1} and \cite{OEIS:vampire2}) for the term
``vampire numbers'', initially introduced by Clifford A. Pickover
\cite{Pickover:vampire}. By one definition, these are Friedman numbers that
make no use of exponentiation.

Another interesting subset,
introduced by Mike Reid, is the ``nice'' Friedman numbers (sometimes referred
to as ``orderly'' or ``good'' Friedman numbers)
\cite{A080035:NiceFriedmanNumbers}. These are the Friedman numbers that can
be calculated from their own digits without changing the digit order.
So, for example, $25=5^2$ is a Friedman number, but it
is not a nice Friedman number. On the other hand, $127=-1+2^7$ is nice.

We wish to answer the following question:
if $F(n)$ is the number of nice Friedman numbers in the range
$[1,n]$, what is $\lim_{n \to \infty} F(n)/n$?

This question, when asked about Friedman numbers in general, was answered in
\cite{Brand:Friedman}. Though the tools given there are not immediately
applicable to the nice Friedman number scenario, we show that they can be
adapted for it. The adaptation proves that the density of nice Friedman numbers
is $1$ for binary, ternary and quaternary nice Friedman numbers. Unfortunately,
it is not strong enough to answer the general case. Specifically, the
interesting case of decimal nice Friedman numbers remains open.

\section{Adapting the technique}

We utilize here the ``infix'' technique introduced in \cite{Brand:Friedman}
and adapt it. (For this, we continue to use terminology that was introduced and
defined there. The reader may want to refer to that other paper for the
definitions.)

The ``infix'' method of \cite{Brand:Friedman} cannot be
used as-is for vampire numbers or for nice Friedman numbers.
For vampire numbers, the infix method
fails because infixes explicitly make use of exponentiation. For
nice Friedman numbers, infixes are constructable using the same method
as shown in \cite{Brand:Friedman}. However, bounding the density based on
the infix method relies on the
following lemma:

\begin{lemma}[\cite{Brand:Friedman}]\label{L:g}
There exists a value of $s$, for which $g=N_b(s)^\frac{s-3}{s}>b^{L(s)}$.
\end{lemma}

This lemma is not true for every base of representation, $b$, once
restricting $\span{s}$ to include only computations that do not change the
order of the digits of $s$.
Where such an $s$ can be found, the previous proof remains intact for
nice Friedman numbers, and shows that their density is $1$.

Let us define $\spant{s}$ similarly to $\span{s}$, but for $\spant{s}$
to include only those elements of the original span
that can be calculated without changing the digit order of $s$, and let us
also define $\tuplespant{\cdot}$ and $N'(\cdot)$ accordingly.

The limitations of the infix method can be expressed as follows:

\begin{thm}\label{T:large_b}
In any base of representation $b \ge 28$,
every $s$ satisfies $b^{L(s)}>|\spant{s}|$.
\end{thm}

This means that even though infixes can be constructed for these bases,
there are too few of them to establish a density of $1$ by the same means
as in \cite{Brand:Friedman}.

\begin{proof}
$\spant{s}$ is the set of numbers that can be produced by a calculation from
the digits of $s$ while retaining the order of $s$'s digits.
Let us consider, equivalently, the set of strings that
describe these calculations. Each character in the strings is
a base-$b$ digit, one of 5 permitted operations or an opening/closing
parenthesis.

In addition to the 5 binary operations we must also consider unary negation.
This adds a sixth operation: $x^{-y}$. In all other operations, when used in
conjunction with unary negation, we can eliminate the unary operation by
pushing it forward in the string until it is voided or until it reaches the
beginning of the string, via operations such as
``$x \times -y$'' $\mapsto$ ``$-(x \times y)$''. After applying these normalization
procedures all inter-digit operations are one of the six binary operations
described, and if there remains a unary negation it must be at the very
beginning of the string.

After the normalization,
any two consecutive digits can have at most one operation
between them, for a total of $6+1=7$ possibilities, and the number of possible
arrangements to surround $l=L(s)$ digits with pairs of parentheses is
the Catalan number $C(l-1)=\binom{2l-2}{l-1}/l$, which is smaller than
$4^l$. The total number of possibilities for such strings is therefore
less than $(7 \times 4)^l$.
\end{proof}

On the other hand, for nice Friedman numbers it is still possible to use
the infix method for small enough values of $b$.

\begin{thm}\label{T:small_b}
In binary, ternary and base four the density of nice Friedman numbers is $1$.
\end{thm}

\begin{proof}
Let us consider $b=2$. Table~\ref{Table:binary} lists the radical-free integers
and the tuple-sizes of such that can be produced from a string of the
form $[1_{BINARY}]^n$. The left-most column gives $n$, followed by the number
$[1_{BINARY}]^n$. The next two columns gives the radical-free integers that
can be produced from this string but not from any smaller $n$, as well as
their total number. Lastly, we give $N'([1_{BINARY}]^n)$.

\begin{table}[H]
\begin{center}
\begin{tabular}{|c|c|c|c|c|}
  \hline
  $n$ & Number & radical-free & total & $N'$ \\
  \hline \hline
  $1$ & $1_{BIN}$ & $\{\}$ & $0$ & $0$ \\
  $2$ & $11_{BIN}$ & $\{2,3\}$ & $2$ & $2$ \\
  $3$ & $111_{BIN}$ & $\{7\}$ & $1$ & $3$ \\
  $4$ & $1111_{BIN}$ & $\{5,6,15\}$ & $3$ & $8$ \\
  $5$ & $11111_{BIN}$ & $\{10,12,14,21,26,28,31\}$ & $7$ & $18$ \\
  $6$ & $111111_{BIN}$ & $\{\ldots\}$ & $23$ & $55$ \\
  $7$ & $1111111_{BIN}$ & $\{\ldots\}$ & $80$ & $170$ \\
  \hline
\end{tabular}
\end{center}
\caption{Radical-free numbers formable from $[1_{BIN}]^n$}
\label{Table:binary}
\end{table}

As can be seen,
the maximal-sized prefix code composed of tuples of
radical-free numbers that can be calculated
from $[1_{BIN}]^7$ is already more than $2^7$, giving an
example of an $s$ for which $N'(s)>b^{L(s)}$ and proving
that the density of binary nice Friedman numbers is $1$.
The prefix code, in this case, is composed of all tuples that can be
calculated from $[1_{BIN}]^7$ but not from $[1_{BIN}]^5$.

To prove for ternary and quaternary numbers, we require a finer tool.
To explain it, we re-prove the claim regarding the density of binary
nice Friedman numbers without making use of the last two rows of
Table~\ref{Table:binary}.

If $X$ is a set of integer tuples, let $\overline{X}$ be the subset of $X$
containing only the tuples composed solely of radical-free integers.

Let
$M(n)$ be $\overline{\tuplespant{[1_{BIN}]^n}} \setminus \overline{\tuplespant{[1_{BIN}]^{n-1}}}$.

Clearly, $N'([1]^n) \ge |M(n)|$, because $M(n)$ is a prefix code.
Furthermore, if $x \in M(n)$ and $y \in M(m)$ then the tuple that is $x$
appended to $y$ is in $M(n+m)$.

Consider, again, the values of the first five rows of Table~\ref{Table:binary}.
The number of radical-free numbers found
for $n=1,\ldots,5$ is $0,2,1,3,7$, respectively.
For any $n>5$, any member of $\overline{\tuplespant{[1]^{n-i}}}$ with $i \le 5$
can be extended by a suffix composed of any member of $\overline{\spant{[1]^i}}$
to make a unique member of $M(n)$.
This indicates that for any $n>5$,
$$N'([1_{BIN}]^n) \ge M(n) \ge 2M(n-2)+M(n-3)+3M(n-4)+7M(n-5).$$

The rate of increase of $N'([1]^n)$ must therefore be in $\Omega(g^n)$, where
$g$ is the greatest real solution to $P(x)=x^5-2x^3-x^2-3x-7$. However,
substituting $x=2$ we get that $P(x)=-1<0$, so the greatest real solution
to $P(x)$ must be greater than $b=2$. If so, then for a sufficiently large
$n$, $N'([1_{BIN}]^n)>b^n$, and using $s=[1_{BIN}]^n$ we can prove that
the density of binary nice Friedman numbers is 1.

This principle can also be applied for ternary and quaternary. In base three
we use $[2_{THREE}]^n$ to get the following number of radical-free numbers in
a partial count of $n=1,\ldots,6$: $1,1,4,22,98,454$. The corresponding
polynomial $P(x)=x^6-x^5-x^4-4x^3-22x^2-98x-454$ has $P(3)=-175<0$ and for
$[3_{FOUR}]^n$ the corresponding polynomial is
$P(x)=x^6-x^5-3x^4-13x^3-59x^2-369x-2279$ with $P(4)=-740<0$. In both cases,
the polynomials have a solution greater than $b$, the base of representation,
proving the claim.
\end{proof}

There is a gap between Theorem~\ref{T:large_b} and Theorem~\ref{T:small_b},
consisting of the bases between $5$ and $27$ inclusive, that still
needs to be fully addressed, but note that even for bases over $27$, the
claims do not immediately imply that the density of nice Friedman numbers is
not $1$.

\bibliographystyle{plain}
\bibliography{bibFriedman}
\end{document}